\documentclass[a4paper, 10pt]{amsart}
\usepackage{amscd,graphicx,epsfig}
\usepackage{amssymb}
\usepackage{mathrsfs}
\usepackage[colorlinks=true, pdfstartview=FitV, linkcolor=blue, citecolor=blue, urlcolor=blue]{hyperref}

\title[Topologies on ind-varieties and irreducibility questions]
{On the topologies on ind-varieties and related irreducibility questions}
\author{Immanuel Stampfli}

\address{Mathematisches Institut,
Universit\"at Basel, Rheinsprung 21, CH-4051 Basel}
\email{Immanuel.Stampfli@unibas.ch}

\thanks{The author is supported by the
Swiss National Science Foundation (Schweizerischer National\-fonds),
grant ``Automorphisms of Affine $n$-Space" 137679}

\theoremstyle{plain}
\newtheorem{thm}{Theorem}
\newtheorem{prop}[thm]{Proposition}
\newtheorem{lem}[thm]{Lemma}
\newtheorem{cor}[thm]{Corollary}

\newtheorem{mainthm}{Theorem}
\newtheorem{mainprop}[mainthm]{Proposition}
\newtheorem{maincor}[mainthm]{Corollary}

\theoremstyle{definition}
\newtheorem{defn}{Definition}
\newtheorem{exa}{Example}

\theoremstyle{remark}
\newtheorem{rem}{Remark}

\newcommand{\name}[1]{\textsc{#1\/}}
\renewcommand{\AA}{{\mathbb A}}
\newcommand{\G}{{\mathcal G}}
\newcommand{\E}{{\mathcal E}}

\renewcommand{\aa}{\mathfrak a}
\newcommand{\mm}{\mathfrak m}
\newcommand{\pp}{\mathfrak p}

\DeclareMathOperator{\GL}{GL}
\DeclareMathOperator{\SL}{SL}
\DeclareMathOperator{\Jac}{Jac}

\newcommand{\OO}{\mathcal O}
\newcommand{\Spm}{\mathfrak{Spm}}
\DeclareMathOperator{\Spec}{Spec}
\DeclareMathOperator{\defneq}{\mathrel{\mathop:}=}
\DeclareMathOperator{\epi}{\twoheadrightarrow}
\DeclareMathOperator{\depth}{depth}
\DeclareMathOperator{\charac}{char}

\newcommand{\lab}[1]{\label{#1}}

\begin{document}

\begin{abstract}
	In the literature there are two ways of endowing an affine ind-variety
	with a topology. One possibility is due to \name{Shafarevich}
	and the other to \name{Kambayashi}. In this paper we
	specify a large class of affine ind-varieties where these two topologies
	differ. We give an example of an affine ind-variety that is reducible with
	respect to \name{Shafarevich}'s topology, but irreducible with
	respect to \name{Kambayashi}'s topology.
	Moreover, we give a counter-example of a supposed irreducibility
	criterion given in \cite{Sh1981On-some-infinite-d} which is different from 
	the counter-example given by \name{Homma} in 
	\cite{Ka1996Pro-affine-algebra}. 
	We finish the paper with an irreducibility criterion similar to the one given by
	\name{Shafarevich}.
\end{abstract}

\maketitle


\setcounter{subsection}{-1}

\subsection{Introduction}
For the sake of simplicity, we assume in this introduction,
that the ground field is algebraically closed, uncountable and
of characteristic zero.

In the 1960s, in \cite{Sh1966On-some-infinite-d}, \name{Shafarevich} introduced 
the notion of an infinite-dimensional variety and infinite-dimensional group.
In this paper, we call them ind-variety and ind-group, respectively. 
His motivation was to explore some naturally occurring groups that allow 
a natural structure of an infinite-dimensional analogue to an algebraic group
(such as the group of polynomial automorphisms of the affine space).
More precisely, he defined an ind-variety as the successive limit
of closed embeddings
\[ 
	X_1 \hookrightarrow X_2 \hookrightarrow X_3 \hookrightarrow \ldots
\] 
of ordinary algebraic varieties $X_n$ and an ind-group as a group 
that carries the structure of an ind-variety
compatible with the group structure. We denote the limit
of $X_1 \hookrightarrow X_2 \hookrightarrow \ldots$ by $\varinjlim X_n$
and call $X_1 \hookrightarrow X_2 \hookrightarrow \ldots$ a \emph{filtration}.
If all $X_n$ are affine, then $\varinjlim X_n$ is called \emph{affine}.
For example, one can define a filtration on the group of polynomial automorphisms 
of the affine space via the degree of an automorphism. Further examples of
ind-groups are $\GL_n(k[t])$, $\SL_n(k[t])$, etc.,  where the filtrations
are given via the degrees of the polynomial entries of the matrices
(for properties of these filtrations
in case $n =2$ see \cite{Sh2004On-the-group-rm-GL}).
Fifteen years after his first paper \cite{Sh1966On-some-infinite-d},
\name{Shafarevich} wrote another paper with the same title \cite{Sh1981On-some-infinite-d}, where he gave 
more detailed explanations of some statements of his first paper. 
Moreover, he endowed an ind-variety $\varinjlim X_n$ with the 
weak topology induced by the topological spaces 
$X_1 \subseteq X_2 \subseteq \ldots \,$.
Later \name{Kambayashi} defined (affine) ind-varieties in 
\cite{Ka1996Pro-affine-algebra} and \cite{Ka2003Some-basic-results} 
via a different approach. Namely, he defined an affine ind-variety as a certain 
spectrum of a so-called pro-affine algebra (see Section~\ref{DefandNot.sec} for the definition).
This pro-affine algebra is then the ring of regular functions on the
affine ind-variety. With this approach 
\name{Kambayashi} introduced a topology in a natural way on an affine ind-variety. 
Namely, a subset is closed if it is the zero-set of some regular functions on the
affine ind-variety. In analogy
to the Zariski topology defined on an ordinary affine variety, we call this topology
again \emph{Zariski topology}. In this paper, we call the weak topology on an 
affine ind-variety \emph{ind-topology} to prevent confusion, as the
weak topology is finer than the Zariski topology. The Zariski topology and
the ind-topology differ in general. 
%
%
%
For example, it follows from Exercise 4.1.E, IV. in \cite{Ku2002Kac-Moody-groups-t}
that these topologies differ on the infinite-dimensional affine space 
$\AA^\infty = \varinjlim \AA^n$ (see Example~\ref{top.exa}). 
The aim of this paper is to specify classes of affine ind-varieties where these
topologies differ or coincide, and to study questions
concerning the irreducibility of an affine ind-variety 
(with respect to these topologies).

This paper is organized as follows. We give some basic definitions and notation
in Section \ref{DefandNot.sec}. In the next section we describe a large class of 
ind-varieties where the two topologies differ. The main result of this paper is 
the following

\begin{mainthm}
	\lab{A.thm}
	Let $X = \varinjlim X_n$ be an affine ind-variety. If there
	exists $x \in X$ such that $X_n$ is normal or Cohen-Macaulay at $x$ 
	for infinitely many $n$, and the local dimension of $X_n$ at $x$
	tends to infinity, then the ind-topology and the Zariski topology are different.
\end{mainthm}

This theorem follows from a more general statement given in 
Proposition~\ref{main.prop} (see also Remark~\ref{main.rem}). 
As a corollary to this theorem we get 

\begin{maincor}
	\lab{B.cor} Let $X = \varinjlim X_n$ be an affine ind-variety such that
	$X_n$ is normal for infinitely many $n$. Then the ind-topology and
	the Zariski topology coincide if and only if for all $x \in X$ the 
	local dimension of $X_n$ at $x$ is bounded for all $n$.
\end{maincor}

This corollary follows from a more general statement given in 
Corollary~\ref{main.cor}. As a contrast to Theorem~\ref{A.thm}, we show in  Proposition~\ref{contrast.prop}
that the two topologies coincide if $X = \varinjlim X_n$ is ``locally constant'' 
with respect to the Zariski topology. More precisely we prove

\begin{mainprop}
	\lab{C.prop}
	If $X = \varinjlim X_n$ is an affine ind-variety such that every point has
	a Zariski open neighbourhood $U$ with
	$U \cap X_{n} = U \cap X_{n+1}$ for all sufficiently large $n$, 
	then the ind-topology and the Zariski topology coincide.
\end{mainprop}

Section \ref{IrredCoord.sec} contains an example of an affine ind-variety that
is reducible with respect to the ind-topology, but irreducible with 
respect to the Zariski topology. This is the content of Example~\ref{irred1.exa}.

In the last section we give a counter-example to 
Proposition~1 in \cite{Sh1981On-some-infinite-d} (see Example \ref{irred2.exa}).
The content of the proposition is: an
ind-variety $X = \varinjlim X_n$ is irreducible with respect to the ind-topology
if and only if the set of irreducible components of all $X_n$ is directed 
under inclusion. One can see that the latter condition is equivalent to the 
existence of a filtration
$X'_1 \hookrightarrow X'_2 \hookrightarrow \ldots$ where each 
$X'_n$ is irreducible and $\varinjlim X'_n = X$.
In \cite{Ka1996Pro-affine-algebra}, \name{Homma} gave
a counter-example to that supposed irreducibility criterion. But in contrast to his 
counter-example, the number of irreducible components of $X_n$
in our counter-example is bounded for all $n$.
We finish the paper with the following irreducibility criterion. 
The proposition follows from Proposition \ref{irred.crit.prop}.

\begin{mainprop}
	\lab{D.prop}
	Let $X = \varinjlim X_n$ be an affine ind-variety 
	where the number of irreducible components of $X_n$ is bounded for all $n$.
	Then $X$ is irreducible with respect to the ind-topology (Zariski topology) 
	if and only if there exists a chain of irreducible subvarieties 
	$X'_1 \subseteq X'_2 \subseteq \ldots$ in $X$ 
	(i.e., $X'_n$ is an irreducible subvariety of some $X_m$) such that 
	$\bigcup_n X'_n$ is dense in $X$ with respect to the ind-topology 
	(Zariski topology).
\end{mainprop}

\subsection{Definitions and notation}
\lab{DefandNot.sec}
Throughout this paper we work over an uncountable algebraically closed field $k$. 
We use the definitions and notation of \name{Kambayashi} in 
\cite{Ka2003Some-basic-results} and \name{Kumar} in
\cite{Ku2002Kac-Moody-groups-t}. Let us recall them briefly.
A \emph{pro-affine algebra} is a complete and separated commutative 
topological $k$-algebra
such that $0$ admits a countable base of open neighbourhoods consisting of
ideals. Let $A$ be a pro-affine algebra and let 
$\aa_1 \supseteq \aa_2 \supseteq \ldots$ be a base for $0 \in A$ 
as mentioned above. Let $A_n = A / \aa_n$ and let $\Spm(A)$ be the set of 
closed maximal ideals of $A$. Then we have
\[
	A = \varprojlim A_n \quad \textrm{and} \quad
	\Spm(A) = \bigcup_{n = 1}^{\infty} \Spm(A_n)
\]
(cf. 1.1 and 1.2 in \cite{Ka2003Some-basic-results}).

\begin{defn}
	An \emph{affine ind-variety} 
	is a pair $(\Spm(A), A)$ where $A$ is a pro-affine algebra such that 
	$A / \aa_n$ is reduced and finitely generated for some
	countable base of ideals $\aa_1 \supseteq \aa_2 \supseteq \ldots$ of 
	$0 \in A$. We call $A$ the \emph{coordinate ring} of the affine ind-variety
	and the elements of $A$ \emph{regular functions}.
	Two ind-varieties are called \emph{isomorphic} if the underlying
	pro-affine algebras are isomorphic as topological $k$-algebras. Such 
	an isomorphism induces then a bijection of the spectra.
\end{defn}

One can construct affine ind-varieties in the following way. Consider a
\emph{filtration of affine varieties}, i.e., a countable 
sequence of closed embeddings of affine varieties
\[
	X_1 \hookrightarrow X_2 \hookrightarrow X_3 \hookrightarrow \ldots \, .
\]
Let $X = \bigcup_{n = 1}^{\infty} X_n$ as a set and
let $\OO(X) \defneq \varprojlim \OO(X_n)$. We endow $\OO(X)$ with the topology
induced by the product topology of $\prod_n \OO(X_n)$, where $\OO(X_n)$
carries the discrete topology for all $n$. Then $(\Spm(\OO(X)), \OO(X))$ 
is an affine ind-variety and there is a natural bijection 
$X \to \Spm(\OO(X))$ induced by the bijections
$X_n \to \Spm(\OO(X_n))$. In the following, we denote this ind-variety by 
$\varinjlim X_n$. In fact, every affine ind-variety can 
be constructed in this way (up to isomorphy).  Two filtrations 
$X_1 \hookrightarrow X_2 \hookrightarrow \ldots$ and 
$X'_1 \hookrightarrow X'_2 \hookrightarrow \ldots$ induce the same affine 
ind-variety (up to isomorphy) if and only if there exists a bijection 
\[
	f \colon \bigcup_{n = 1}^\infty X_n \to \bigcup_{n = 1}^\infty X'_n
\]
with the following property: for every $i$ there exists $j_i$ and for every
$j$ there exists $i_j$, such that $f |_{X_i} \colon X_i \to X_{j_i}$ and
$f^{-1} |_{X_j} \colon X_j \to X_{i_j}$ are closed embeddings of affine varieties.
Such filtrations are called \emph{equivalent}.



\subsection{Topologies on affine ind-varieties}
\lab{Topo.sec}

So far we have not established any topology on the set $\Spm(A)$
of an affine ind-variety $(\Spm(A), A)$. As mentioned in
the introduction there are two ways to introduce a 
topology on the set $\Spm(A)$. 
The first possibility is due to \name{Shafarevich} 
\cite{Sh1966On-some-infinite-d}, \cite{Sh1981On-some-infinite-d}
and we call it the \emph{ind-topology}. A subset $Y \subseteq \Spm(A)$
is closed in this topology if and only if $A \cap \Spm(A_n)$ is a closed
subset of $\Spm(A_n)$ for all $n$. One can easily check that
this topology does not depend on the choice of the ideals 
$\aa_1 \subseteq \aa_2 \subseteq \ldots \,$.
The second possibility is due to \name{Kambayashi} 
\cite{Ka1996Pro-affine-algebra}, \cite{Ka2003Some-basic-results} and we call 
it the \emph{Zariski topology}. The closed subsets in this topology are the 
subsets of the form
\[
	V_{\Spm(A)}(E) \defneq \{ \, \mm \in \Spm(A) \ | \ \mm \supseteq E \, \} \, ,
\]
where $E$ is any subset of $A$. Clearly, the ind-topology
is finer than the Zariski topology. But in general
these two topologies on $\Spm(A)$ differ. In the next 
proposition (which implies Theorem~\ref{A.thm}) we specify a large class
of affine ind-varieties where the two topologies differ. 

\begin{prop}
	\lab{main.prop}
	We assume that $\charac(k)=0$.
	Let $X = \varinjlim X_n$ be an affine ind-variety. Assume that there exists
	$x \in X$ such that $\OO_{X_n, x}$ satisfies
	\name{Serre}'s condition $(S_2)$ for infinitely many $n$ 
	and $\dim_x X_n \to \infty$ if $n \to \infty$. 
	Then there exists a subset $Y \subseteq X$ such that
	\begin{enumerate} 
	\item[i)] $Y$ is closed in $X$ with respect to the ind-topology,	
	\item[ii)] $Y$ is not closed in $X$ with respect to the Zariski topology.
	\end{enumerate}
	In particular, there exist no isomorphism $X \to X$ of affine ind-varieties that
	is a homeomorphism if we endow the first $X$ with the ind-topology 
	and the second $X$ with the Zariski topology.
\end{prop}

\begin{rem}
	\lab{main.rem}
	A Noetherian ring $A$ satisfies
	\name{Serre}'s condition $(S_2)$ if
	$\depth A_{\pp}  \geq \min \{ \dim A_{\pp}, 2 \}$
	for all primes $\pp \subseteq A$. For example, this is satisfied if
	$A$ is normal (and hence also if $A$ is a unique factorization domain) 
	or Cohen-Macaulay 
	(and hence also if $A$ is Gorenstein, locally a complete intersection or regular) 
	(see Theorem~23.8 \cite{Ma1986Commutative-ring-t}).
\end{rem}

We will use the following lemmata to prove Proposition~\ref{main.prop}.

\begin{lem}
	\lab{finite.lem}
	Let $Z$ and $Y$ be affine varieties and assume that there exists a 
	closed embedding $Z \hookrightarrow Y$. If $f \colon Z \epi \AA^{\dim Z}$
	is a finite surjective morphism, then there exists a finite 
	surjective morphism $g \colon Y \epi \AA^{\dim Y}$ such that 
	$g |_{Z} = \iota \circ f$, where 
	$\iota: \AA^{\dim Z} \hookrightarrow \AA^{\dim Y}$ is given by 
	$\iota(v) = (v, 0)$.
\end{lem}

\begin{proof}
	Let $A \defneq \OO(Z)$, $B \defneq \OO(Y)$
	and let $\psi \colon B \epi A$ be the surjective homomorphism
	induced by $Z \hookrightarrow Y$. Further, let $f_1, \ldots, f_n$ be the 
	coordinate functions of $f$. By assumption
	$k[f_1, \ldots, f_n] \subseteq A$ is an integral extension and
	$f_1, \ldots, f_n$ are algebraically independent.
	Choose generators $b_1, \ldots, b_l$ of the $k$-algebra $B$
	such that $\psi(b_i) = f_i$ for $i = 1, \ldots, n$. For every $j = n+1, \ldots, l$
	there exists a monic polynomial
	$p_j \in k[b_1, \ldots, b_n][T]$ such that 
	$h_j \defneq p_j(b_j) \in \ker(\psi)$, since $k[f_1, \ldots, f_n] \subseteq A$ 
	is integral. Thus,
	\[
		k[b_1, \ldots, b_n, h_{n+1}, \ldots, h_l] \subseteq B
	\]
	is an integral extension. If $b_1, \ldots, b_n, h_{n+1}, \ldots, h_l$
	are algebraically independent, then we are done. Otherwise, there exists 
	a non-zero polynomial $f(X_1, \ldots, X_l)$ with coefficients in $k$ such that 
	$f(b_1, \ldots, b_n, h_{n+1}, \ldots, h_l) = 0$. 
	Exactly the same as in the proof of Lemma~2, 
	\S33 \cite{Ma1986Commutative-ring-t} one can see that 
	there exist $c_1, \ldots, c_{l-1} \in k$ such that $h_l$ is integral 
	over $k[b'_1, \ldots, b'_{n}, h'_{n+1}, \ldots, h'_{l-1}]$, where
	$b'_i \defneq b_i - c_i h_l$ and $h'_i \defneq h_i - c_i h_l$. Thus, 
	\[
		k[b'_1, \ldots, b'_n, h'_{n+1}, \ldots, h'_{l-1}] \subseteq B
	\]
	is an integral extension. By induction there exists $n \leq m < l$ and
	algebraically independent elements
	$b''_1, \ldots, b''_n, h''_{n+1}, \ldots, h''_m \in B$ such that $B$ is integral 
	over $k[b''_1, \ldots, b''_n, h''_{n+1}, \ldots, h''_m]$ and 
	$b''_i - b_i, \, h''_i \in \ker(\psi)$. This proves the lemma.
\end{proof}

\begin{rem}
	From an iterative use of the lemma above we can deduce the following.
	For every affine ind-variety $X = \varinjlim X_n$ there exists
	a surjective map of the underlying sets $X \epi \AA^{\infty}$ 
	such that the restriction to every $X_n$ yields a finite surjective morphism 
	$X_n \epi \AA^{\dim X_n}$.
\end{rem}

\begin{lem}
	\lab{gen.red.fibers.lem}
	We assume that $\charac(k) = 0$.
	Let $Y$ be an irreducible affine variety and let $X$ be an
	affine scheme of finite type over $k$ that is reduced in an 
	open dense subset. If $f \colon X \to Y$ is a dominant morphism, 
	then there exists an open dense subset $U \subseteq Y$ 
	such that $f^{-1}(u)$ is reduced in an open dense subset for all $u \in U$.
\end{lem}

\begin{proof}
	Without loss of generality, one can assume that $f$ is flat 
	and surjective (see Theorem~14.4 (Generic freeness) 
	\cite{Ei1995Commutative-algebr}). Since $X$ is reduced in an open dense
	subset, there exists an open dense subset $X' \subseteq X$
	such that all fibers of $f |_{X'} \colon X' \to Y$ are reduced 
	(see Corollary~10.7, Ch.~III (Generic smoothness) and Theorem~10.2, Ch.~III 
	\cite{Ha1977Algebraic-geometry}; here we use $\charac(k)=0$). 
	Let $K \defneq X \setminus X'$ be endowed with the reduced induced
	closed subscheme structure of $X$ and let $g \defneq f |_K \colon K \to Y$. 
	If $g$ is not dominant, then the fibers
	of $f$ over an open dense subset are reduced and we are done. 
	Hence we can assume that $g$ is dominant. 
	Again according to Theorem~24.1 \cite{Ma1986Commutative-ring-t} there
	exists an open dense subset $U \subseteq Y$ such that
	$g |_{g^{-1}(U)} \colon g^{-1}(U) \epi U$ is flat and surjective. 
	Thus, we have for all $u \in U$ and $x \in g^{-1}(u)$
	\[
		\dim_x g^{-1}(u) = \dim_x g^{-1}(U) - \dim_u U
		< \dim_x X - \dim_u Y = \dim_x f^{-1}( u ) \, .
	\]
	It follows that $f^{-1}(u) \setminus g^{-1}(u)$ is a reduced open dense 
	subscheme of $f^{-1}(u)$ for all $u \in U$. This implies the lemma.
\end{proof}

According to Ex. 11.10 \cite{Ei1995Commutative-algebr} we have the 
following criterion for reducedness of a Noetherian ring.

\begin{lem}
	\lab{red.crit.lem}
	A Noetherian ring $A$ is reduced if and only if
	\begin{itemize}
		\item[$(R_0)$] the localization of $A$ at each prime ideal of 
				       height $0$ is regular,
		\item[$(S_1)$] $A$ has no embedded associated prime ideals.
	\end{itemize}
\end{lem}

One can see that condition $(R_0)$ is satisfied for
a Noetherian ring $A$ if $\Spec(A)$ is reduced in an open dense subset.
Thus we get the following 

\begin{lem}
	\lab{red.lem}
	Let $X$ be a Noetherian affine scheme that is
	reduced in an open dense subset. If $\OO_{X, x}$
	satisfies $(S_1)$ for a point $x \in X$, then $\OO_{X, x}$ is reduced.
\end{lem}



Now we have the preliminary results to prove  
Proposition~\ref{main.prop}. The strategy is as follows. First
we construct $0 \neq f_n \in \OO(X_n)$ such that
$f_n(x) = 0$, $f_n |_{X_{n-1}} = f_{n-1}^2$ and 
$\OO_{X_n, x} / f_n \OO_{X_n, x}$ is reduced. 
The main part of the proof is to show the reducedness and for that matter
we use the condition $(S_2)$ of
the local ring $\OO_{X_n, x}$. Then we define
$Y \defneq \bigcup_n V_{X_n}(f_n)$. It follows that $Y$ is closed in $X$
with respect to the ind-topology.
Afterwards, we prove that $Y$ is not closed in $X$
with respect to the Zariski topology.
For that purpose, we take 
$\varphi = (\varphi_n) \in \OO(X) = \varprojlim \OO(X_n)$ that
vanishes on $Y$, and we show that $\varphi_n$ vanishes
also on all irreducible components of $X_n$ passing through $x$. 
The latter we deduce from the fact that 
\[
	\varphi_n = \varphi_{n+i} |_{X_n} \in f_{n+i} |_{X_n} \OO_{X_n, x} 
	= f_n^{2^i} \OO_{X_n, x}
\]
for all $i \geq 0$ and \name{Krull}'s Intersection Theorem. 

\begin{proof}[Proof of Proposition~\ref{main.prop}]
	For the sake of simpler notation, we assume that
	$\OO_{X_n, x}$ satisfies $(S_2)$ and $\dim_x X_n = n$ for all $n$. 
	Let $X'_n$ be the union of all irreducible components of $X_n$ 
	containing $x$ and let $W_n$ be the union of all irreducible components 
	of all $X_i$ with $i \leq n$, not containing $x$ and of strictly smaller 
	dimension than $n$. Then, 
	$X'_1 \cup W_1 \hookrightarrow  X'_2 \cup W_2 \hookrightarrow \ldots$
	is an equivalent filtration of $X$ to 
	$X_1 \hookrightarrow X_2 \hookrightarrow \ldots \,$, since 
	$\dim_x X_n \to \infty$. Thus, we can further impose that 
	$\dim_x X_n = \dim X'_n = \dim X_n$ and $\dim_p X_n < \dim X_n$
	for all $p \notin X'_n$. As $\OO_{X_n, x}$ satisfies $(S_2)$, it follows
	from Corollary~5.10.9 \cite{Gr1965Elements-de-geomet-IV} that
	$X'_n$ is equidimensional.
	
	Now, we construct the $0 \neq f_n \in \OO(X_n)$. 
	From Lemma~\ref{finite.lem} it follows that we can choose
	algebraically independent elements $x_1, \ldots, x_n \in \OO(X_n)$ 
	such that $\OO(X_n)$ is finite over $k[x_1, \ldots, x_n]$ and 
	$x_n$ restricted to $X_{n-1}$ is zero. We can assume that the 
	finite morphism $X_n \epi \AA^n$ induced by
	$k[x_1, \ldots, x_n] \subseteq \OO(X_n)$ sends $x$ to $0 \in \AA^n$. Since 
	$\dim_p X_n < \dim X_n$ for all $p \notin X'_n$ and $X'_n$ is 
	equidimensional, it follows that 
	\begin{equation}
		\label{injectivity}
		\tag{$\ast$}
		k[x_1, \ldots, x_n] \hookrightarrow \OO(X_n) \epi
		\OO(K) \quad  \textrm{is injective}
	\end{equation}
	for all irreducible components $K$ of $X'_n$. Let us define
	\[
		f_1 \defneq  c_1 x_1 \, \quad{and} \quad 
		f_{n+1} \defneq f_n^2 + c_{n+1} x_{n+1} \, , 
	\]
	where $c_1, c_2, \ldots \in k$, not all equal to zero. It follows that 
	$f_n(x) = 0$ and $f_{n+1} |_{X_n} = f_n^2$. The aim is to prove that
	$c_1, c_2, \ldots \in k$ can be chosen such that not all are equal to zero
	and  $\OO_{X_n, x} / f_n \OO_{X_n, x}$ is reduced for $n > 1$. 
	Consider the morphism
	\[
		\psi_nÊ\colon Z_n \longrightarrow \AA^n \, ,
	\]
	where $Z_n$ is the affine scheme with coordinate ring
	\[
		S_n \defneq \OO(X'_n)[c_1, \ldots, c_n] / (f_n)
	\]
	and $\psi_n$ is the restriction of the canonical projection 
	$X'_n \times \AA^n \epi \AA^n$ to the closed subscheme $Z_n$.
	If $(c_1, \ldots, c_n)$ is fixed, then $\OO_{X_n, x} / f_n \OO_{X_n, x}$
	is the local ring of the fiber $\psi_n^{-1}(c_1, \ldots, c_n)$ in the point
	$(x, c_1, \ldots, c_n) \in Z_n$. For that reason we will study the fibers of
	the morphism $\psi_n \colon Z_n \to \AA^n$.
	We claim that $Z_n$ is reduced in an open dense 
	subset for $n > 1$. To prove this claim, we mention first that
	\[
		(S_n)_{x_n} 
		\simeq \OO(X'_n)_{x_n}[c_1, \ldots, c_n] / (f_{n-1}^2 + c_n x_n) 
		\simeq \OO(X'_n)_{x_n}[c_1, \ldots, c_{n-1}]
	\] 
	is reduced. Let $R_n \defneq k[x_1, \ldots, x_n][c_1, \ldots, c_n] / (f_n)$.
	It follows that the morphisms $\Spec(S_n) \epi \Spec(R_n)$ and 
	$\Spec(S_n / (x_n)) \epi \Spec(R_n / (x_n))$ are both finite and surjective.
	As $\dim R_n / (x_n) < \dim R_n$ for $n > 1$ we get
	$\dim S_n / (x_n) < \dim S_n$. 
	Since $X'_n$ is equidimensional
	one can deduce from (\ref{injectivity}) that $Z_n$ is 
	equidimensional. Hence, $\Spec((S_n)_{x_n}) \subseteq Z_n$
	is an open dense reduced subscheme.

	Since $\{ x \} \times \AA^n$ is contained in $Z_n$, it follows that 
	$\psi_n$ is surjective. 
	For $n > 1$ there exists an open dense subset 
	$U_n \subseteq \AA^n$ such that 
	\[
		\psi_n |_{\psi_n^{-1}(U_n)} \colon \psi_n^{-1}(U_n)
		\epi U_n
	\]
	is surjective and flat, and every fiber is reduced in an open dense subset
	(see Lemma~\ref{gen.red.fibers.lem}
	and Theorem~24.1 \cite{Ma1986Commutative-ring-t}).
	With the aid of (\ref{injectivity}) it follows that $f_n$ is an 
	$\OO_{X_n, x}$-regular sequence for every choice $(c_1, \ldots, c_n) \in U_n$.
	Since $\OO_{X_n, x}$ satisfies $(S_2)$, we get
	from Corollary~5.7.6 \cite{Gr1965Elements-de-geomet-IV}
	that $\OO_{X_n, x} / f_n \OO_{X_n, x}$ satisfies $(S_1)$.
	But as $\psi_n^{-1}(c_1, \ldots, c_n)$ is reduced in an open dense subset, 
	it follows from Lemma~\ref{red.lem} that it is reduced in 
	the point $(x, c_1, \ldots, c_n)$. Hence, for $n > 1$ it follows that
	$\OO_{X_n, x} / f_n \OO_{X_n, x}$ is reduced if we choose 
	$(c_1, \ldots, c_n) \in U_n$.
	For $i \geq n$ let $\pi_n^i: \AA^i \epi \AA^n$ be the projection onto 
	the first $n$ components. As the field $k$ is uncountable, 
	one can choose inductively
	\[
		0 \neq c_1 \in \bigcap_{i \geq 1} \pi_1^i(U_i) \, , \quad
		(c_1, \ldots, c_n, c_{n+1}) \in 
		\bigcap_{i \geq n+1} 
		\pi_{n+1}^i( U_i) \cap \{(c_1, \ldots, c_n)\} \times \AA^1 \, .
	\]
	Hence, $(c_1, \ldots, c_n) \in U_n$ for all $n > 1$ and not all 
	$c_1, c_2, \ldots$ are equal to zero.
	This finishes the construction of the $f_n$.
	
	Let us define $Y \defneq \bigcup_n V_{X_n}(f_n)$.
	Since $f_{n+1} |_{X_n} = f_n^2$ for all $n$, $Y$ satisfies i).
	Take any $\varphi = (\varphi_n) \in \varprojlim \OO(X_n)$ that vanishes
	on $Y$. We claim that $\varphi |_{X'} = 0$, where
	$X' \defneq \bigcup_n X'_n$. It is enough to prove that
	$\varphi_n = 0$ in $\OO_{X_n, x}$. Since
	$\varphi_m |_{Y_m} = 0$ and
	$\OO_{X_m, x} / f_m \OO_{X_m, x}$ is reduced, 
	it follows that $\varphi_m \in f_m \OO_{X_m, x}$. 
	Using $f_{m+1} |_{X_m} = f_m^2$ again, we get by induction
	\[
		\varphi_n = \varphi_{n + i} |_{X_n}
		\in f_n^{2^i} \OO_{X_n, x} \quad \textrm{for all $i \geq 0$} \, , \  n > 1 \, .
	\]
	But according to \name{Krull}'s Intersection Theorem
	(see Theorem~8.10 \cite{Ma1986Commutative-ring-t}),
	we have $\bigcap_{i \geq 0} f_n^{i} \OO_{X_n, x} = 0$, hence
	$\varphi_n = 0$ in $\OO_{X_n, x}$. 
	Since $f_n |_{X'_n} \neq 0$ (cf. (\ref{injectivity})), we get
	$X' \cup Y \supsetneq Y$. Thus $Y$ satisfies ii) according to the 
	afore mentioned claim.
\end{proof}

The following example is a special case of the construction in the proof
of Proposition~\ref{main.prop}. We mention it here, since we will use it in
future examples.
	
\begin{exa}[cf. Ex. 4.1.E, IV. in \cite{Ku2002Kac-Moody-groups-t}]
	\lab{top.exa}
	Let $f_n \in k[x_1, \ldots, x_{n}] = \OO(\AA^n)$ be recursively defined as
	\[
		f_1 \defneq x_1 \, , \quad f_{n+1} \defneq f_n^2 + x_{n+1} \, . 
	\]
	Then $\bigcup_n V_{\AA^n}(f_n)$ is a proper closed subset of the 
	infinite-dimensional affine space $\AA^\infty = \varinjlim \AA^n$ with respect to 
	the ind-topology, but it is dense in $\AA^\infty$ with respect to 
	the Zariski topology. 
\end{exa}

Let $\G$ be the group of polynomial automorphisms
of the affine space $\AA^n$, where $n$ is a fixed number $\geq 2$.
We prove in the next example that the ind-topology and the 
Zariski topology on $\G$ differ if we consider $\G$ as an affine ind-variety via 
the filtration given by the degree of an automorphism.

\begin{exa}
	\lab{auto.exa}
	First, we define on $\G$ a filtration of affine varieties (via the degree).
	Let $\E$ be the set of polynomial endomorphisms of the affine 
	space $\AA^n$ and let $\E_d$ be the subset of all 
	$\varphi \in \E$ of degree $\leq d$.
	Denote by $U_d \subseteq \E_d$
	the subset of all $\varphi \in \E_d$ such that $\Jac(\varphi) \in k^\ast$.
	One can see that $U_d \subseteq \E_d$ is a locally closed subset
	and it inherits the structure of an affine variety from $\E_d$.
	With Corollary 0.2 \cite{Ka1979Automorphism-group} and the estimate 
	of the degree of the inverse of an automorphism due to \name{Gabber} 
	(see Corollary 1.4 in \cite{BaCoWr1982The-Jacobian-conje})
	one can deduce that $\G_d \subseteq U_d$ is a closed subset. Thus 
	$\G_d$ is locally closed in $\E_d$ and it
	inherits the structure of an affine variety from $\E_d$. Moreover, one
	can see that $\G_d$ is closed in $\G_{d+1}$. 
	In the following, we consider $\G$ as an affine ind-variety
	via the filtration $\G_1 \hookrightarrow \G_2 \hookrightarrow \ldots \,$
	of affine varieties.
	
	We claim that the ind-topology and the Zariski topology on $\G$
	differ. Consider the subset
	\[
		M \defneq \{ \, (x_1 + p, x_2, \ldots, x_n) \in \G \ | \ p \in k[x_n] \, \} 
		\subseteq \G \, .
	\]
	It is closed in $\G$ with respect to the ind-topology.
	We consider $M$ as an affine ind-variety via 
	$M \defneq \varinjlim M \cap G_d$ and thus 
	$M \simeq \AA^\infty$ as affine ind-varieties.
	According to Example \ref{top.exa} there exists a proper 
	subset $Y \subsetneq M$ that is closed with respect to the 
	ind-topology, but it is dense in $M$
	with respect to the Zariski topology. Hence, every regular function
	on $\G$ vanishing on $Y$, vanishes also on $M$. This implies the claim.
\end{exa}

\begin{rem}
	A similar argument as in Example \ref{auto.exa} shows that
	the ind-topology and the Zariski topology differ on $\GL_n(k[t])$
	and also on $\SL_n(k[t])$.
\end{rem}

Next, we give a corollary to Proposition~\ref{main.prop} which implies
Corollary~\ref{B.cor}. Before we state the corollary, we introduce the
following notation. For any ind-variety $X = \varinjlim X_n$ we choose connected
components $X_n^{i}$ of $X_n$, $i = 1, \ldots, k_n$, such that
\[
	X_n = \bigcup_{i = 1}^{k_n} X_n^{i} \quad \textrm{and} \quad
	X_n^{i} \subseteq X_{n+1}^{i} \ \textrm{for all $i = 1, \ldots, k_n$}
\]
(it can be that $X_n^{i} = X_n^{j}$ for $i \neq j$). We remark that
the decomposition of an ind-variety into connected components is the
same for the ind-topology and the Zariski topology.

\begin{cor}
	\lab{main.cor}
	We assume that $\charac(k)=0$.
	Let $X = \varinjlim X_n$ be an affine ind-variety such that for $i$ fixed, 
	the number of irreducible components of $X_n^{i}$ 
	is bounded for all $n$. Moreover, assume that
	$\OO(X_n)$ satisfies $(S_2)$ for infinitely many $n$. Then the following
	statements are equivalent:
	\begin{itemize}
		\item[i)] The ind-topology and the Zariski topology on $X$ coincide.
		\item[ii)] For all $x \in X$ the local dimension of $X_n$ at $x$
		is bounded for all $n$.
		\item[iii)] Every connected component of $X$ is contained in some
		$X_n$.
	\end{itemize}
\end{cor}

\begin{proof}
	Every connected component of $X$ is equal to some
	$X^i \defneq \bigcup_n X_n^i$. \\
	i) $\Rightarrow$ ii): This follows from Proposition \ref{main.prop}. \\
	ii) $\Rightarrow$ iii): As $X_n^i$ satisfies $(S_2)$ and is connected,
	$X_n^i$ is equidimensional 
	(see Corollary~5.10.9 \cite{Gr1965Elements-de-geomet-IV}). Thus, 
	$X_n^i = X_{n+1}^i$ for $n$ large enough, as the number of irreducible
	components of $X_n^i$ is bounded for all $n$. Thus, $X^i \subseteq X_n$ 
	for some $n$. \\
	iii) $\Rightarrow$ i): This follows from the fact that every 
	connected component of $X$ is closed and open with respect to the 
	Zariski topology.
\end{proof}

As a contrast to Proposition~\ref{main.prop}, the two topologies coincide 
if the affine ind-variety is ``locally constant'' with respect to the Zariski topology.
The following proposition coincides with Proposition~\ref{C.prop}.

\begin{prop}
	\lab{contrast.prop}
	Let $X = \varinjlim X_n$ be an affine ind-variety. 
	Assume that every $x \in X$ has a Zariski 
	open neighbourhood $U_x \subseteq X$ 
	such that $U_x \cap  X_n = U_x \cap X_{n+1}$ for all sufficiently large $n$. 
	Then the two topologies on $X$ coincide.
\end{prop}

\begin{proof}
	Let $Y \subseteq X$ be a closed subset with respect to the 
	ind-topology. One can see that $Y \cap U_x$ is closed in $U_x$ with 
	respect to the Zariski topology for all $x \in X$. This proves that $Y$
	is closed in $X$ with respect to the Zariski topology.
\end{proof}

The following example is an application of the proposition above. We construct
a proper ind-variety (i.e., it is not a variety) such that the ind-topology
and the Zariski topology coincide and moreover, it is connected.

\begin{exa}
	Let $L_n$ be defined as
	\[
		L_n \defneq V_{\AA^n}(x_1 - 1, x_2 - 1, \ldots, x_{n-1} - 1)
		\subseteq \AA^n \, .
	\]
	Remark that $L_n \cap L_{n+1} = \{ (1, \ldots, 1) \} \subseteq \AA^n$ 
	for all $n$ and
	$L_n \cap L_m = \varnothing$ for all $n$, $m$ with $|n-m| \geq 2$.
	Let $X \defneq \varinjlim X_n$ where 
	$X_n \defneq L_1 \cup \ldots \cup L_n \subseteq \AA^n$.
	It follows that $X \subseteq \AA^{\infty}$ is a closed connected 
	subset in the ind-topology. We claim that
	the ind-topology and the Zariski topology on $X$ coincide. According
	to the proposition above it is enough
	to show that $X$ is ``locally constant'' with respect to the Zariski topology.
	Let $x \in X$. Then there exists $N$ such that $x \in L_N$, but 
	$x \notin L_{N+1}$. Let 
	$U_x \defneq X \setminus V_{\AA^\infty}(f_1, \ldots, f_N) \subseteq X$
	where $f_i \in \OO(\AA^\infty)$ is given by
	\[
		  f_i |_{\AA^n} = x_i - 1 \in k[x_1, \ldots, x_n]
		  \quad \textrm{for all $n \geq N$}.
	\]
	Thus, $U_x \subseteq X$ is a Zariski open neighbourhood of $x$. 
	Moreover, for all $n > N$ we have 
	$L_n \subseteq V_{\AA^\infty}(f_1, \ldots, f_N)$. Hence we have
	$U_x \cap X_n = U_x \cap X_{n+1}$ for all $n \geq N$.
\end{exa}

As remarked before Corollary~\ref{main.cor}, 
connectedness of an affine ind-variety is the same for both topologies. But this
is no longer true for irreducibility as we will see
in the next section (see Example~\ref{irred1.exa}).

\subsection{Irreducibility via the coordinate ring}
\lab{IrredCoord.sec}

It is well known that an affine variety $X$ is irreducible if and only if
the coordinate ring $\OO(X)$ is an integral domain. 
This statement remains true for affine ind-varieties endowed
with the Zariski topology. The proof is completely analogous
to the proof for affine varieties. 
In the case of the ind-topology it is still true that
$\OO(X)$ is an integral domain if $X$ is irreducible, as the ind-topology
is finer than the Zariski topology.
But the converse is in general false. In the following we give an example
of an affine ind-variety $X$, which is reducible in the ind-topology, but
its coordinate ring $\OO(X)$ is an integral domain and thus it is irreducible
in the Zariski topology.

\begin{exa}
	\lab{irred1.exa}
	Throughout this example we work in the ind-topology.
	Let $g_n \in k[x_1, \ldots, x_{n}]$ be defined as
	\[
		g_n \defneq x_1 + \ldots + x_n \, ,
	\]
	and let $f_n$ be defined as in Example~\ref{top.exa}.
	By construction, $f_n$ and $g_n$ are irreducible polynomials.
	The affine ind-variety 
	$X \defneq \varinjlim (V_{\AA^n}(f_n) \cup V_{\AA^n}(g_n))$
	decomposes into the proper closed subsets $\bigcup_n V_{\AA^n}(f_n)$
	and $\bigcup_n V_{\AA^n}(g_n))$ and thus $X$ is reducible.
	We claim that $\OO(X) = \varprojlim k[x_1, \ldots, x_n] / (f_n g_n)$ 
	is an integral domain. Assume towards a contradiction that there exist 
	$(\varphi_n), (\psi_n) \in \prod_{n=1}^{\infty} k[x_1, \ldots, x_n]$ such that
	$(\varphi_n)$ and $(\psi_n)$ define non-zero elements in $\OO(X)$, but 
	$(\varphi_n \psi_n)$ defines zero in $\OO(X)$. 
	By definition, there exists $\alpha_n \in k[x_1, \ldots, x_n]$ such that
	\begin{equation}
		\label{eq1}
		\tag{$\ast$}
		\varphi_{n+1}(x_1, \ldots, x_n, 0) = \varphi_n + f_n g_n \alpha_n 
		\quad \textrm{for all $n$} \, .
	\end{equation}
	Since $(\varphi_n \psi_n)$ defines zero in $\OO(X)$, it 
	follows that $f_n g_n$ divides $\varphi_n \psi_n$ for $n > 0$.
	Hence we can assume without loss of generality that 
	$f_n$ divides $\varphi_n$ for
	infinitely many $n$. The equation (\ref{eq1}) and the 
	definition of $f_{n+1}$ shows that $f_n$ divides $\varphi_n$ for all $n$. 
	Since $(\varphi_n) \neq 0$ in $\OO(X)$ there exists $N > 1$ such that 
	$g_N$ does not divide $\varphi_N$.
	Let $\rho_n \in k[x_1, \ldots, x_n]$ such that $\varphi_n = f_n \rho_n$. 
	It follows that $g_N$ does not divide $\rho_N$, in particular $\rho_N \neq 0$.
	According to (\ref{eq1}) and the definition of $f_{n+1}$ we have
	\begin{equation}
		\label{eq2}
		\tag{$\ast \ast$}
		\rho_n = f_n \cdot \rho_{n+1}(x_1, \ldots, x_n, 0) - g_n \cdot \alpha_n 
		\quad \textrm{for all $n$} \, .
	\end{equation}
	Since $g_N$ does not divide $\rho_N$ it follows that there exists 
	$p \in \AA^N$ with $g_N(p) = 0$ and $\rho_N(p) \neq 0$. 
	Let $\gamma_n \colon \AA^1 \to \AA^n$ be the curve defined by 
	$\gamma_n(t) = (p, 0, \ldots, 0) + (t, -t, 0, \ldots, 0)$ for $n \geq N$.
	Since $g_n(\gamma_n(t)) = 0$ it follows from (\ref{eq2})
	that $\rho_n(\gamma_n(t)) = f_n(\gamma_n(t)) \rho_{n+1}(\gamma_{n+1}(t))$.
	This implies
	\[
		0 \neq 
		\rho_N(\gamma_N(t)) = \left( \prod_{i = N}^{n-1} f_i(\gamma_i(t)) \right)
		\cdot
		\rho_n(\gamma_n(t)) 
		\quad \textrm{for all $n \geq N$} \, .
	\]
	Since $f_i(\gamma_i(t))$ is a polynomial of degree $2^{i-1}$ for all 
	$i \geq N$, it follows that the polynomial 
	$\rho_N(\gamma_N(t))$ is of unbounded degree, a contradiction. 
\end{exa}

\subsection{Irreducibility via the filtration}
\lab{IrredFilt.sec}

One would like to give a criterion for connectedness or irreducibility 
in terms of the filtration $X_1 \hookrightarrow X_2 \hookrightarrow \ldots$ 
of the affine ind-variety.
In the case of connectedness \name{Shafarevich} gave a nice description
via the filtration (see Proposition~2, \cite{Sh1981On-some-infinite-d}) 
and \name{Kambayashi} 
gave a proof for it (see Proposition~2.4, \cite{Ka1996Pro-affine-algebra})
(the proof works in both topologies, as connectedness of an affine 
ind-variety is the same for both topologies). 
In the case of irreducibility, things look different.

If we start with a filtration $X_1 \hookrightarrow X_2 \hookrightarrow \ldots$
of irreducible affine varieties, then one can see that
$\varinjlim X_n$ is an irreducible affine ind-variety in both topologies. 
Likewise one can ask if every irreducible affine ind-variety is obtained from 
a filtration of irreducible affine varieties.
One can see that the latter property is equivalent to the following
condition: the set $\mathscr{K}$ of all irreducible 
components of all $X_n$ is directed under inclusion
for some (and hence every) filtration 
$X_1 \hookrightarrow X_2 \hookrightarrow \ldots \,$.
\name{Shafarevich} claims in \cite{Sh1981On-some-infinite-d} 
that the latter condition is equivalent to the irreducibility of $X$ in the ind-topology.
But \name{Homma} gave in \cite{Ka1996Pro-affine-algebra} a counter-example 
$X$ to this statement. For every filtration 
$X_1 \hookrightarrow X_2 \hookrightarrow \ldots$ of \name{Homma}'s 
counter-example $X$ the number of irreducible components of 
$X_n$ tends to infinity if $n \to \infty$. Here we give 
another counter-example. Namely, we construct an irreducible affine ind-variety 
$X = \varinjlim X_n$ (irreducible with respect to both topologies)
such that $\mathscr{K}$ is not directed, but $X_n$ consist of exactly 
two irreducible components for $n > 1$.

\begin{exa}
	\lab{irred2.exa}
	Let us define $g_n \in k[x_1, \ldots, x_n]$ recursively by
	\[
		g_1 \defneq (x_1 - 1) \, , \qquad 
		g_{n+1} \defneq (x_1 - (n+1)) \cdot g_n - x_{n+1} \, .
	\]
	By construction every $g_n$ is an irreducible polynomial. Let 
	$Y_n \defneq V_{\AA^n}(g_n) \subseteq \AA^n$. It follows that 
	$Y_n \subseteq Y_{n+1}$ for all $n$. Let further 
	$Z_n \defneq V_{\AA^n}(x_2, \ldots, x_n) \subseteq \AA^n$ and 
	$X_n \defneq Y_n \cup Z_n$. It follows that $X_n \subseteq X_{n+1}$
	is a closed subset for all $n$. Let $X \defneq \varinjlim X_n$. 
	We get
	\[
		Y_n \cap Z_n = V_{\AA^n}(g_n, x_2, \ldots, x_n) = 
		V_{\AA^n}( \prod_{i=1}^n (x_1 - i), x_2, \ldots, x_n) = 
		\{ e_1, 2 e_1, \ldots, n e_1 \} \, ,
	\]
	where $e_1 = (1, 0, \ldots, 0) \in \AA^n$. 
	The set $\mathscr{K}$ defined above is not directed and 
	$X_n$ decomposes in two 
	irreducible components for $n > 1$.
	It remains to show that $X$ is irreducible with respect
	to the ind-topology, as in that case $X$ is also irreducible
	in the Zariski topology. As $Y_n$ is irreducible for all 
	$n$, it follows that $Y = \bigcup_n Y_n$ is irreducible. Since
	\[
		Z_m \subseteq \overline{ \bigcup_{n=1}^{\infty} Y_n \cap Z_n } \subseteq 
		\overline{Y} \subseteq X \qquad \textrm{for all $m$} \, ,
	\]
	we have $X = \overline{Y}$, where the closure is taken in the ind-topology. 
	Since $Y$ is irreducible, as a consequence $X$ is also irreducible.
\end{exa}

We conclude this paper with a criterion for the irreducibility of an 
affine ind-variety $X = \varinjlim X_n$ where the number of irreducible
components of $X_n$ is bounded for all $n$. Unfortunately we need for this
criterion also information about the closure of a subset in the ``global'' 
object $X$ and not only about the filtration 
$X_1 \hookrightarrow X_2 \hookrightarrow~\ldots \,$ itself.
The following proposition implies Proposition~\ref{D.prop}.

\begin{prop}
	\lab{irred.crit.prop}
	Let $X = \varinjlim X_n$ be an affine ind-variety such that
	the number of irreducible components of $X_n$ is bounded by $l$ 
	for all $n$. Then $X$ is irreducible in the ind-topology (Zariski topology)
	if and only if for all $n$ there exists an irreducible component
	$F_n$ of $X_n$ such that $F_1 \subseteq F_2 \subseteq \ldots$
	and $\bigcup_n F_n$ is dense in $X$ with respect to the 
	ind-topology (Zariski topology).
\end{prop}

\begin{proof}
	One can read the proof either with respect to the ind-topology or 
	with respect to the Zariski topology. Let $X = \varinjlim X_n$ be irreducible. 
	For all $n$ let us write $X_n = X^1_n \cup \ldots \cup X^l_n$ 
	where $X^{i}_n$ is an irreducible component of $X_n$ and for all $n$
	we have $X^{i}_n \subseteq X^{i}_{n+1}$ 
	(it can be that $X^i_n = X^j_n$ for $i \neq j$). Thus, one get
	\[
		X = \overline{\bigcup_{n} X^1_n} \cup \ldots \cup 
		\overline{\bigcup_{n} X^l_n} \, .
	\]
	Since $X$ is irreducible the claim follows.
	The converse of the statement is clear.
\end{proof}

\par\bigskip\bigskip
\vskip 0.5cm

\bibliographystyle{amsalpha}

\providecommand{\bysame}{\leavevmode\hbox to3em{\hrulefill}\thinspace}
\providecommand{\MR}{\relax\ifhmode\unskip\space\fi MR }
\providecommand{\MRhref}[2]{%
  \href{http://www.ams.org/mathscinet-getitem?mr=#1}{#2}
}
\providecommand{\href}[2]{#2}

\vskip .5cm

\end{document}